\documentclass[11pt,twoside]{article}
\usepackage{mathrsfs}
\usepackage{amsmath}
\usepackage{amsthm}
\input amssymb.sty
\usepackage{amsfonts}
\usepackage{amssymb}
\usepackage{latexsym}
\usepackage[all]{xy}

\date{\empty}
\allowdisplaybreaks[4] \footskip=15pt

\numberwithin{equation}{section} \theoremstyle{plain}
\newtheorem*{thm*}{Main Theorem}
\newtheorem{theorem}{Theorem}
\newtheorem{corollary}[theorem]{Corollary}
\newtheorem*{corollary*}{Corollary}

\newtheorem*{claim*}{Claim}
\newtheorem{lemma}[theorem]{Lemma}
\newtheorem*{lemma*}{Lemma}

\newtheorem*{proposition*}{Proposition}

\newtheorem*{remark*}{Remark}
\newtheorem{example}[theorem]{Example}
\newtheorem*{example*}{Example}

\newtheorem*{question*}{Question}

\newtheorem*{definition*}{Definition}

\begin{document}
\begin{center}
{\large \bf Moore-Penrose Invertibility of Differences and Products of Projections in Rings with Involution}\\

\vspace{0.8cm} {\small \bf Xiaoxiang Zhang, \ \ Shuangshuang Zhang, \ \ Jianlong Chen\footnote{Corresponding author.\\
E-mail: z990303@seu.edu.cn (X. Zhang), jlchen@seu.edu.cn (J. Chen).},\ \ Long Wang}\\
\vspace{0.6cm} {\rm Department of Mathematics, Southeast University, Nanjing, 210096, China}
\end{center}

\bigskip

{ \bf  Abstract:}
\leftskip0truemm\rightskip0truemm
This article concerns the MP inverse of the differences and the products of projections in a ring $R$ with involution.
Some equivalent conditions are obtained.
As applications, the MP invertibility of the commutator $pq-qp$ and the anti-commutator $pq+qp$ are characterized,
where $p$ and $q$ are projections in $R$.
Some related known results in $C^*$-algebra are generalized. 
 \bigskip


\section { \bf Introduction}
Moore-Penrose inverse (abbr. MP inverse) of the sums, differences and the products of projections in various settings attracts wide interest from many authors.
For instance, Cheng and Tian \cite{Cheng Tian} presented expressions for the MP inverse of such matrices as $P_AP_B$, $P_A - P_B$, and $P_AP_B - P_BP_A$,
where $A$ and $B$ are complex matrices, $P_A = AA^{\dag}$ and $P_B = BB^{\dag}$.
Du and Deng \cite{Du Deng} studied the MP inverses of $PQ$ and $P-Q$,
where $P$ and $Q$ are projections in the ring $B(H)$ of all linear bounded operators on a Hilbert space $H$.
Li \cite{Li 2008, Li 2009} investigated MP inverses of $pq$, $p - q$, $pq - qp$ and $pq + qp$
for two given projections $p$ and $q$ in a $C^*$-algebra.
Recently, Deng and  Wei \cite{Deng Wei} established some formulae for the MP inverse of the sums, differences and the products of projections in a Hilbert space
so that some known results in the literature were extended.

This article is mainly motivated by \cite{Deng Wei, Li 2008, Li 2009}.
We investigate the MP inverse of the differences and the products of projections in a ring $R$ with involution.
Some equivalent conditions are obtained.
As applications, the MP invertibility of the commutator $pq-qp$ and the anti-commutator $pq+qp$ are characterized,
where $p$ and $q$ are projections in $R$.
Some results in \cite{Deng Wei, Li 2008, Li 2009} are generalized.
Note that some methods based on decomposition of matrix, orthogonal decomposition of Hilbert spaces and spectral technique
are not available in an arbitrary ring with involution.
The results in this paper are proved by a purely ring theoretical method.

Throughout this paper, $R$ is an associative ring with unity and an involution $a \mapsto a^*$ satisfying
$(a^*)^* = a$, $(a+b)^* = a^* + b^*$, $(ab)^* = b^*a^*$.
An element $a\in R$ has MP inverse,
if there exists $b$ such that the following equations hold \cite{Penrose}:
$$(1)\ \ aba = a, \quad\quad  (2)\ \ bab = b, \quad\quad  (3)\ \ (ab)^* = ab, \quad\quad (4)\ \ (ba)^{\ast} = ba.$$
In this case, $b$ is unique and denoted by $a^{\dag}$.
Moreover, we have $(a^{\dag})^{\dag} = a$.
It is easy to see that $a$ has MP inverse if and only if $a^{*}$ has MP inverse, and in this case $(a^{*})^{\dag} = (a^{\dag})^{*}$.

An element $a\in R$ has Drazin inverse,
if there exists $b$ such that the following equations hold \cite{Drazin}:
$$(1)\ \ ab = ba, \quad\quad  (2)\ \ bab = b, \quad\quad  (3)\ \ a^{k+1}b = a^{k} .$$
for some nonnegative integer $k$.
In this case, $b$ is unique and denoted by $a^{\mathrm{d}}$.
The smallest integer $k$ for which the above equations hold is called the Drazin index
of $a$, denoted by ind$(a)$. If $k = 1$, then $b$ is called group inverse of $a$ and denoted by $a^{\sharp}$.

We write $R^{\dag}$, $R^{\mathrm{d}}$ and $R^{\sharp}$ as the set of all MP invertible elements,
all Drazin invertible elements and group invertible elements in $R$, respectively.

If $a^{*} = a\in R^{\dag}$, then $aa^{\dag} = a^{\dag}a$, i.e. $a^{\dag} = a^{\sharp}$.
In this case, $a$ is called an EP element.
An idempotent $p \in R$ is called a projection if it is self-adjoint, i.e., $p^{*} = p$.

Recall from \cite{Koliha Patricio} that a ring $R$ is said to be $*$-reducing if, for any element $a\in R$,
$a^*a = 0$ implies $a = 0$.
Note that $R$ is $*$-reducing if and only if the following implications hold for any $a\in R$:
$$a^{*}ax = a^{*}ay \Rightarrow ax = ay \quad\quad \mbox{and} \quad\quad xaa^{*} = yaa^{*} \Rightarrow xa = ya.$$
It is well-known that any $C^{*}$-algebra is a $*$-reducing ring.

\section{ \bf  Main results}

Let us remind the reader that, in what follows, $R$ is always a ring with involution $*$.
By $p$ and $q$ we mean two projections in $R$.
We also fix the notations $a = pqp$, $b = pq(1-p)$, $d = (1-p)q(1-p)$, $\overline{p} = 1-p$ and $\overline{q} = 1-q$.

The following lemmas will be used in the sequel.

\begin{lemma}\label{Lemma 2.1}\emph{(\cite[Lemma 2.1]{Benitez CvetkovicIlic})}
\emph{(1)} If $r\in R^{\dag}$, then $r^{*}r$, $rr^{*}\in R^{\dag}$ and the following equalities hold:
           $(r^{*}r)^{\dag} = r^{\dag}(r^{*})^{\dag}$,
           $(rr^{*})^{\dag} = (r^{*})^{\dag}r^{\dag}$,
           $r^{\dag} = (r^{*}r)^{\dag}r^{*} = r^{*}(rr^{*})^{\dag}.$

\emph{(2)} Suppose that $R$ is $*$-reducing.
           If $r^{*}r\in R^{\dag}$ or $rr^{*}\in R^{\dag}$, then $r\in R^{\dag}$.
\end{lemma}

\begin{lemma}\label{Lemma 2.2}
\emph{(1)} $bb^{*} = (p-a)-(p-a)^{2}$.

\emph{(2)}  $b^{*}b =d-d^{2}$.

\emph{(3)}  $db^{*} = b^{*}(p-a)$.
\end{lemma}
\begin{proof}
By a direct verification.
\end{proof}

The next lemma is essentially due to Benitez and Cvetkovic-Ilic \cite{Benitez CvetkovicIlic}
although we do not assume that $R$ is $*$-reducing.

\begin{lemma}\label{Lemma 2.3}
\emph{(1)} If $p\overline{q}\in R^{\dag}$, then $(p-a)(p-a)^{\dag}b = b$;

\emph{(2)} If $\overline{p}q\in R^{\dag}$, then $bdd^{\dag} = b$;

\emph{(3)} If $p\overline{q}, \overline{p}q\in R^{\dag}$, then $bd^{\dag} = (p-a)^{\dag}b$ and $d^{\dag}b^{*} = b^{*}(p-a)^{\dag}$.

\emph{(4)} If $p\overline{q}, \overline{p}q\in R^{\dag}$, then $p-q \in R^{\dag}$ and $(p-q)^{\dag} = \overline{q}(p\overline{q}p)^{\dag}-q(\overline{p}q\overline{p})^{\dag}$.
\end{lemma}
\begin{proof}
(1)-(3) See \cite[Lemma 2.3]{Benitez CvetkovicIlic}.

(4) If $p\overline{q}, \overline{p}q\in R^{\dag}$, then we have
    \begin{eqnarray}
           (p-q)^{\dag} & = & (p-a)(p-a)^{\dag}-bd^{\dag}-d^{\dag}b^{*}-dd^{\dag}                                   \nonumber \\
                        & \overset{(3)}{=\!\!\!=\!\!\!=} & (p-a)(p-a)^{\dag}-bd^{\dag}-b^*(p-a)^{\dag} -dd^{\dag}   \nonumber \\
                        & = & [(p-a)- b^{*}](p-a)^{\dag}-(b+d)d^{\dag}                                              \nonumber \\
                        & = & [p(1-q)p+(1-p)(1-q)p](p\overline{q}p)^{\dag}-q(1-p)(\overline{p}q\overline{p})^{\dag} \nonumber \\
                        & = & \overline{q}(p\overline{q}p)^{\dag}-q(\overline{p}q\overline{p})^{\dag},
    \end{eqnarray}
    where the first ``='' follows by the proof of \cite[Theorem 4.1(iii)]{Benitez CvetkovicIlic}.
\end{proof}

The following theorem and its corollaries parallel to \cite[Theorem 2.3]{Deng Wei}.

\begin{theorem}\label{Theorem 2.4}
The following statements are equivalent for any two projections $p$ and $q$ in a ring $R$ with involution:

\emph{(1)} $1-pq\in R^{\dag}$,  \quad \emph{(2)} $1-pqp\in R^{\dag}$,  \quad \emph{(3)} $p-pqp\in R^{\dag}$;

\emph{(4)} $1-qp\in R^{\dag}$,  \quad \emph{(5)} $1-qpq\in R^{\dag}$,  \quad \emph{(6)} $q-qpq\in R^{\dag}$.
\end{theorem}

\begin{proof}
(1)$\Leftrightarrow$(4) is clear by $1-qp = (1-pq)^{*}$.
                        We need only prove that (1)-(3) are equivalent.

(1)$\Rightarrow$(3) Let $x = (1-pq)^{\dag}$.
                    Then we have
                    \begin{eqnarray}
                          (1-p)x(1-pq)p = (1-p)(1-pq)x(1-pq)p = (1-p)(1-pq)p = 0.
                    \end{eqnarray}
                    Consequently,
                    \begin{eqnarray}
                          px(1-pq)(1-p) = [(1-p)x(1-pq)p]^* = 0.
                    \end{eqnarray}
                    On the other hand, by a direct verification, we have
                    \begin{eqnarray}
                          p(1-pq)xp & = & p(1-pq)[p+(1-p)]xp         \nonumber\\
                                    & = & p(1-pq)pxp+p(1-pq)(1-p)xp  \nonumber\\
                                    & = & (p - pqp)pxp - pq(1-p)xp.
                    \end{eqnarray}
                    Now, we verify $p - pqp \in R^{\dag}$ and $(p - pqp)^{\dag} = pxp$.

                    \underline{\emph{Step 1}}. $(p-pqp)pxp(p-pqp) = p-pqp$.
                    Indeed,
                    \begin{eqnarray}
                           (p-pqp)pxp(p-pqp) & \overset{(2.2)}{=\!\!\!=\!\!\!=\!\!\!=} & (p-pqp)pxp(1-pq)p-pq[(1-p)x(1-pq)p] \nonumber \\
                                             & =                                       & [(p-pqp)pxp - pq(1-p)xp](1-pq)p      \nonumber \\
                                             & \overset{(2.4)}{=\!\!\!=\!\!\!=\!\!\!=} & p(1-pq)xp(1-pq)p                  \nonumber \\
                                             & =                                       & p(1-pq)x(1-pq)p                   \nonumber \\
                                             & =                                       & p(1-pq)p
                                               =                                         p-pqp.
                    \end{eqnarray}

                    \underline{\emph{Step 2}}. $[(p-pqp)pxp]^{*} = (p-pqp)pxp$.

                    Actually, (2.2) implies $(1-p)x^{*}p - (1-p)qpx^{*}p$ = $(1-p)(1-qp)x^{*}p$ = $(1-p)[x(1-pq)]^{*}p$ = $(1-p)x(1-pq)p$ = 0.
                    Hence
                    \begin{eqnarray}
                          (1-p)x^{*}p = (1-p)qpx^{*}p.
                    \end{eqnarray}
                    Meanwhile, $(p-pqp)px^*p(p-pqp)$ = $[(p-pqp)pxp(p-pqp)]^* \overset{(2.5)}{=\!\!\!=\!\!\!=\!\!\!=}$ $[p-pqp]^*$ = $p-pqp$, i.e.,
                    \begin{eqnarray}
                          (p-a)px^*p(p-a) = p-a.
                    \end{eqnarray}

                    In view of (2.6), (2.7) and Lemma \ref{Lemma 2.2}(1), one can get
                    \begin{eqnarray*}
                           pq(1-p)xp & = & pq(1-p)(1-pq)xp = pq(1-p)x^{*}(1-qp)p \\
                                     & = & pq(1-p)x^{*}p(1-qp)p+pq(1-p)x^{*}(1-p)(1-qp)p \\
                                     & \overset{(2.6)}{=\!\!\!=\!\!\!=\!\!\!=} & pq(1-p)qpx^{*}p(1-q)p-pq(1-p)x^{*}(1-p)qp \\
                                     & = & bb^{*}x^{*}(p-a)-pq(1-p)x^{*}(1-qp)(1-p)qp \\
                                     & = & [(p-a)-(p-a)^{2}]x^{*}(p-a) - b(1-pq)xb^* \quad\quad (\mbox{see Lemma \ref{Lemma 2.2}}(1))\\
                                     & \overset{(2.7)}{=\!\!\!=\!\!\!=\!\!\!=} & (p-a)-(p-a)^{2}- b(1-pq)xb^*,
                    \end{eqnarray*}
                    where $(p-a)^* = p-a$ and $(b(1-pq)xb^*)^* = b(1-pq)xb^*$.
                    This guarantees
                    \begin{eqnarray}
                           [pq(1-p)xp]^{*} = pq(1-p)xp.
                    \end{eqnarray}
                    So we have
                    \begin{eqnarray*}
                           [(p-pqp)pxp]^{*} & \overset{(2.4)}{=\!\!\!=\!\!\!=\!\!\!=} & [p(1-pq)xp + pq(1-p)xp]^{*}   \\
                                            & \overset{(2.8)}{=\!\!\!=\!\!\!=\!\!\!=} & p(1-pq)xp + pq(1-p)xp \\
                                            & \overset{(2.4)}{=\!\!\!=\!\!\!=\!\!\!=} & (p-pqp)pxp.
                    \end{eqnarray*}

                    \underline{\emph{Step 3}}. $[pxp(p-pqp)]^*$ = $[px(1-pq)p]^*$ = $px(1-pq)p$ = $pxp(p-pqp)$.

                    \underline{\emph{Step 4}}. $pxp(p-pqp)pxp = pxp$. Indeed,
                    \begin{eqnarray*}
                           pxp(p-pqp)pxp & \overset{(2.3)}{=\!\!\!=\!\!\!=\!\!\!=} & px(1-pq)pxp + [px(1-pq)(1-p)]xp \\
                                         & = & px(1-pq)[p+(1-p)]xp \\
                                         & = & px(1-pq)xp = pxp.
                    \end{eqnarray*}

(3)$\Rightarrow$(1) Suppose $p-a = p-pqp \in R^{\dag}$. Let
                    $$x = [1 + b^*(p-a)](p-a)^{\dag}(1+b) - b^* - b^*b + 1-p.$$
                    We will show that $(1-pq)^{\dag} = x$ in accordance with the definition of MP inverse.

                    Firstly, note that
                    \begin{eqnarray*}
                          &   & (1-pq)x                                                                         \\
                          & = & (1-pq)p(p-a)^{\dag}(1+b)+(1-pq)b^*(p-a)(p-a)^{\dag}(1+b)   \\
                          &   &  -(1-pq)b^* -(1-pq)b^*b+(1-pq)(1-p)                            \\
                          & = & (p-a)(p-a)^{\dag}(1+b)+b^*(p-a)(p-a)^{\dag}(1+b) \\
                          &   &  -bb^*(p-a)(p-a)^{\dag}(1+b) -b^*+bb^*-b^*b+bb^*b+1-p-b,
                    \end{eqnarray*}
                    where $bb^* = (p-a)-(p-a)^{2}$ (see Lemma \ref{Lemma 2.2}(1)).
                    Hence
                    \begin{eqnarray}
                          (1-pq)x & = & (p-a)(p-a)^{\dag}+(p-a)(p-a)^{\dag}b+b^*(p-a)(p-a)^{\dag}  \nonumber \\
                                  &   &  +b^*(p-a)(p-a)^{\dag}b-b^*b-b^*-b+1-p.
                    \end{eqnarray}
                    Now, it is straightforward to check
                    \begin{eqnarray}
                           [(1-pq)x]^{*} = (1-pq)x.
                    \end{eqnarray}

                    Similarly, we have
                    \begin{eqnarray}
                           &   & x(1-pq)                                                             \nonumber \\
                           & = & (p-a)^{\dag}(1-pq+b)+b^*(p-a)(p-a)^{\dag}(1-pq+b) -b^*(1-pq+b)+1-p  \nonumber \\
                           & = & (p-a)^{\dag}(p-a)+b^*(p-a)-b^*(1-pqp)+1-p                           \nonumber \\
                           & = & (p-a)^{\dag}(p-a) + 1-p,
                    \end{eqnarray}
                    from which it is easy to see that
                    \begin{eqnarray}
                           [x(1-pq)]^{*} = x(1-pq).
                    \end{eqnarray}

                    Secondly, it follows from (2.9) that
                    \begin{eqnarray}
                           &   & (1-pq)x(1-pq)                                                            \nonumber\\
                           & = & (p-a)(p-a)^{\dag}(1-pq+b)+b^*(p-a)(p-a)^{\dag}(1-pq+b)                   \nonumber\\
                           &   &  -b^*(1-pq+b)-b+1-p                                                   \nonumber\\
                           & = & (p-a)+b^*(p-a)-b^*(p-a)-b+1-p                                      \nonumber\\
                           & = & p-a-b+1-p = 1-pq.
                    \end{eqnarray}

                    Finally, we have
                    \begin{eqnarray}
                          &   & x(1-pq)x                                                                        \nonumber\\
                          & \overset{(2.11)}{=\!\!=\!\!=\!\!=} & [(p-a)^{\dag}(p-a)+1-p]\{[1 + b^*(p-a)](p-a)^{\dag}(1+b) - b^* - b^*b + 1-p\}   \nonumber\\
                          & = & [1 + b^*(p-a)](p-a)^{\dag}(1+b) - b^* - b^*b + 1-p = x.
                    \end{eqnarray}

                    Combining (2.10), (2.12)-(2.14), one can see that
                    \begin{eqnarray}
                           (1-pq)^{\dag} = x = [1 + b^*(p-a)](p-a)^{\dag}(1+b) - b^* - b^*b + 1-p.
                    \end{eqnarray}

(2)$\Rightarrow$(3) Since $(1-pqp)^{*} = 1-pqp\in R^{\dag}$ and  $p(1-pqp) = (1-pqp)p$,
                    we have $p(1-pqp)^{\dag} = (1-pqp)^{\dag}p$ by \cite[Corollary 12]{Mary}.
                    One can check $(p-pqp)^{\dag} = p(1-pqp)^{\dag}$.

(3)$\Rightarrow$(2) It is trivial to verify that $(1-pqp)^{\dag} = (p-pqp)^{\dag}+1-p$.
\end{proof}

\begin{corollary}\label{corollary 2.5}
The following conditions are equivalent for any two projections $p$ and $q$ in a $*$-reducing ring $R$:

\emph{(1)} $1-pq\in R^{\dag}$,  \emph{(2)} $1-pqp\in R^{\dag}$, \emph{(3)} $p-pqp\in R^{\dag}$, \emph{(4)} $p-pq \in R^{\dag}$,
\emph{(5)} $p-qp\in R^{\dag}$;

\emph{(6)} $1-qp\in R^{\dag}$,  \emph{(7)} $1-qpq\in R^{\dag}$, \emph{(8)} $q-qpq\in R^{\dag}$, \emph{(9)} $q-qp \in R^{\dag}$,
\emph{(10)} $q-pq\in R^{\dag}$.\\
Moreover, $(p-pqp)^{\dag} = (1-pq)^{\dag}p$ when any one of these conditions is satisfied.
\end{corollary}
\begin{proof}
(1)$\Leftrightarrow$(2)$\Leftrightarrow$(3)$\Leftrightarrow$(6) has been proved in Theorem \ref{Theorem 2.4}.
We will prove (3)$\Leftrightarrow$(4)$\Leftrightarrow$(5).

(4)$\Leftrightarrow$(5) is obvious since $p-pq = (p-qp)^{*}$.

(3)$\Leftrightarrow$(4) Since $R$ is $*$-reducing and $p-pqp = p(1-q)p = p\overline{q}(p\overline{q})^{*}$,
                        it is easy to see that $p(1-q)\in R^{\dag}$ if and only if $p-pqp\in R^{\dag}$ by Lemma \ref{Lemma 2.1}.

Moreover, if any one of the above conditions is satisfied, then we have
\begin{eqnarray*}
       (1-pq)^{\dag}p & \overset{(2.15)}{=\!\!\!=\!\!\!=\!\!\!=\!\!\!=} & \{[1 + b^*(p-a)](p-a)^{\dag}(1+b) - b^* - b^*b + 1-p\}p                            \\
                      & = & (p-a)^{\dag}+b^*(p-a)(p-a)^{\dag}-b^*                            \\
                      & = & (p-a)^{\dag} = (p-pqp)^{\dag}. \quad\quad\quad\quad\quad\quad (\mbox{see Lemma \ref{Lemma 2.3}(1)})
\end{eqnarray*}
This completes the proof.
\end{proof}

By Lemma \ref{Lemma 2.1}(1),
one can see that (4)$\Rightarrow$(3) and (5)$\Rightarrow$(3) in Corollary \ref{corollary 2.5} are valid even if $R$ is not $*$-reducing.
However, the following example shows that (3) does not imply (4) or (5) in general if $R$ is not $*$-reducing.

\begin{example}\label{conterexample}
\emph{Let $R = \mathbb{Z}_{2}\langle x, y\rangle/(x^{2}-x, y^{2}-y, xyx)$
be the ring generated over $\mathbb{Z}_{2}$ by $\{x, y\}$ with the relations $\{x^{2}-x, y^{2}-y, xyx\}$.
Let $X = x+(x^{2}-x, y^{2}-y, xyx)$ and $Y = y+(x^{2}-x, y^{2}-y, xyx)$.
Define the involution on $R$ such that $1^* = 1$, $X^{*} = X$, $Y^{*} = Y$,
$(XY)^* = YX$, $(YX)^* = XY$ and
$(YXY)^* = YXY$.
Then $p = X$ and $q = 1- Y$ are projections.
In addition, $p(1-q)p = XYX = 0 \in R^{\dag}$.
But $p(1-q) = XY\not\in p(1-q)[p(1-q)]^*R = \{0\}$.
Therefore, $p(1-q)\not\in R^{\dag}$ in according to \cite[Theorem 5.4]{Koliha Patricio}.}
\end{example}

Replacing $p$ and $q$ by $\overline{p} = 1-p$ and $\overline{q} = 1-q$ respectively in Corollary \ref{corollary 2.5},
we obtain the following corollary.

\begin{corollary}\label{Corollary 2.6}
The following are equivalent for any two projections $p$ and $q$ in a $*$-reducing ring $R$:

$\begin{array}{llll}
(1)\ p+q-pq\in R^{\dag},        & (2)\ p+\overline{p}q\overline{p}\in R^{\dag}, & (3)\ \overline{p}q\overline{p}\in R^{\dag},   & (4)\ q-pq\in R^{\dag},  \\
(5)\ q-qp\in R^{\dag},          & (6)\ p+q-qp\in R^{\dag},                      & (7)\ q+\overline{q}p\overline{q}\in R^{\dag}, & (8)\ \overline{q}p\overline{q}\in R^{\dag}, \\
(9)\ p-qp\in R^{\dag},          & (10)\ p-pq\in R^{\dag}.                       &                                               &
\end{array}$
\end{corollary}

\begin{theorem}\label{Theorem 2.7}
The following statements are equivalent for any two projections $p$ and $q$ in a ring $R$ with involution:

\emph{(1)} $p(1-q)\in R^{\dag}$ and $(1-p)q\in R^{\dag}$;

\emph{(2)} $p-q\in R^{\dag}$.
\end{theorem}

\begin{proof}
(1)$\Rightarrow$(2) By Lemma \ref{Lemma 2.3}(4).

(2)$\Rightarrow$(1) It follows that $(p-q)^{\dag}(p-q) = (p-q)(p-q)^{\dag}$ since $p-q\in R^{\dag}$ and $(p-q)^{*} = p-q$.
                    Now, it is easy to check that $[(p-q)^{2}]^{\dag} = [(p-q)^{\dag}]^{2}$.
                    Also note that $[(p-q)^{2}]^{*} = (p-q)^{2}$ and $p(p-q)^{2} = (p-q)^{2}p = p-pqp$.
                    By \cite[Corollary 12]{Mary}, we have $p[(p-q)^{2}]^{\dag} = [(p-q)^{2}]^{\dag}p$,
                    i.e., $p[(p-q)^{\dag}]^{2} = [(p-q)^{\dag}]^{2}p$.
                    Let $x = (p-q)^{\dag}p$.
                    We will prove that $(p\overline{q})^{\dag} = x$.

                    First, $(p\overline{q})x = p(p-q)(p-q)^{\dag}p$ implies $[(p\overline{q})x]^{*} = (p\overline{q})x$.

                    Moreover, we have
                    \begin{eqnarray*}
                           (p\overline{q})x(p\overline{q}) & = & p(p-q)(p-q)^{\dag}pp(p-q) \\
                                                           & = & p(p-q)(p-q)(p-q)^{\dag}(p-q)^{\dag}pp(p-q)\\
                                                           & = & p(p-q)(p-q)(p-q)^{\dag}(p-q)^{\dag}(p-q) \\
                                                           & = & p(p-q) = p\overline{q}.
                    \end{eqnarray*}

                    Furthermore, $x(p\overline{q}) = (p-q)^{\dag}p(p-q)$ implies
                    \begin{eqnarray*}
                           [x(p\overline{q})]^{*} & = & [(p-q)^{\dag}p(p-q)]^* = (p-q)p(p-q)^{\dag}  \\
                                                  & = & (p-q)p(p-q)^{\dag}(p-q)^{\dag}(p-q)  \\
                                                  & = & (p-q)(p-q)^{\dag}(p-q)^{\dag}p(p-q)                      \\
                                                  & = & (p-q)^{\dag}p(p-q) = x(p\overline{q}),
                    \end{eqnarray*}
                    and hence
                    \begin{eqnarray*}
                           x(p\overline{q})x & = & [x(p\overline{q})]^{*}x = [(p-q)^{\dag}p(p-q)]^*(p-q)^{\dag}p \\
                                             & = &(p-q)p(p-q)^{\dag}(p-q)^{\dag}p \\
                                             & = & (p-q)(p-q)^{\dag}(p-q)^{\dag}p \\
                                             & = & (p-q)^{\dag}p = x.
                    \end{eqnarray*}

                    This proves $(p\overline{q})^{\dag} = (p-q)^{\dag}p$. \hfill (2.16)

                    Thus, $\overline{p}-\overline{q} = -(p-q)\in R^{\dag}$ implies $\overline{p}q\in R^{\dag}$.
\end{proof}

In case $R$ is $*$-reducing, $p(1-q)\in R^{\dag}$ if and only if $(1-p)q\in R^{\dag}$ by Corollary \ref{corollary 2.5}(3)$\Leftrightarrow$(9).
Whence we have the following corollary.

\begin{corollary}\label{corollary 2.8}
The following conditions are equivalent for any two projections $p$ and $q$ in a $*$-reducing ring $R$:

\emph{(1)} $p(1-q)\in R^{\dag}$;

\emph{(2)} $p-q\in R^{\dag}$;

\emph{(3)} $(1-p)q\in R^{\dag}$.
\end{corollary}

Let $p$ and $q$ be projections in a $*$-reducing ring $R$.
It is clear that all the conditions in Corollary \ref{corollary 2.5}, \ref{Corollary 2.6} and \ref{corollary 2.8}
are mutually equivalent.
The following corollary shows the relations among those MP inverses when they exist.

\begin{corollary}\label{Corollary 2.9}
Let $p$ and $q$ be projections in a $*$-reducing ring $R$.
If any one of the conditions in Corollary \emph{\ref{corollary 2.5}, \ref{Corollary 2.6}} and \emph{\ref{corollary 2.8}}
is satisfied, then

\emph{(1)}  $(1-pq)^{\dag}p\overline{q}p = p\overline{q}(p\overline{q}p)^{\dag} = p\overline{q}p(1-qp)^{\dag}
                                         = p(p\overline{q})^{\dag} = (\overline{q}p)^{\dag}p = p(p-q)^{\dag}p$;

\emph{(2)} $(p+\overline{p}q)^{\dag}\overline{p}q\overline{p} = (q+p\overline{q})^{\dag}\overline{p}q\overline{p}
                                                              = \overline{p}q(\overline{p}q\overline{p})^{\dag}
                                                              = \overline{p}q\overline{p}(p+q\overline{p})^{\dag}
                                                              = \overline{p}q\overline{p}(q+\overline{q}p)^{\dag}
                                                              = \overline{p}(\overline{p}q)^{\dag}
                                                              = (q\overline{p})^{\dag}\overline{p}
                                                              =  \overline{p}(q-p)^{\dag}\overline{p}$.
\end{corollary}

\begin{proof}
(1) First of all, we have $(p\overline{q}p)^{\dag} = (p-pqp)^{\dag} = (1-pq)^{\dag}p$ by Corollary \ref{corollary 2.5}.
    Hence
    $$(1-pq)^{\dag}p\overline{q}p = (p\overline{q}p)^{\dag}p\overline{q}p
                                  = p\overline{q}p(p\overline{q}p)^{\dag}
                                  = p\overline{q}(p\overline{q}p)^{\dag}. $$
    Consequently,
    $$p\overline{q}(p\overline{q}p)^{\dag} =  p\overline{q}p(p\overline{q}p)^{\dag}
                                           = [p\overline{q}p(p\overline{q}p)^{\dag}]^{*}
                                           = [(1-pq)^{\dag}p\overline{q}p]^{*} = p\overline{q}p(1-qp)^{\dag}.$$

    Next, it follows by (2.1) that
    $$p\overline{q}(p\overline{q}p)^{\dag} = p\overline{q}(p\overline{q}p)^{\dag}p
                                           = p[\overline{q}(p\overline{q}p)^{\dag}-q(\overline{p}q\overline{p})^{\dag}]p
                                           \overset{(2.1)}{=\!\!\!=\!\!\!=\!\!\!=} p(p-q)^{\dag}p.$$
    Finally, from (2.16) we can see that
    $p(p\overline{q})^{\dag} = p(p-q)^{\dag}p = (\overline{q}p)^{\dag}p.$

(2) Replace $p$ and $q$, respectively, by $1-p$ and $1-q$ in (1).
\end{proof}

In order to investigate the MP invertibility of the commutator $pq-qp$ and the anti-commutator $pq+qp$,
we need the following three lemmas.

\begin{lemma}\label{lemma 2.10}
The following conditions are equivalent for any two projections $p$ and $q$ in a $*$-reducing ring $R$:

\emph{(1)} $(1-p)(1-q)\in R^{\dag}$;

\emph{(2)} $1-p-q\in R^{\dag}$;

\emph{(3)} $pq\in R^{\dag}$.
\end{lemma}
\begin{proof}
Substitute $1-p$ for $p$ in Corollary \ref{corollary 2.8}.
\end{proof}

\begin{lemma}\label{Lemma 2.11}
Let $b = pq(1-p)$, where $p$ and $q$ be two projections in $R$,
then $b-b^*\in R^{\mathrm{d}}$ if and only if $bb^*\in R^{\mathrm{d}}$.
In this case, $\mathrm{ind}(bb^{*})$ $\leq$ $\mathrm{ind}[(b-b^{*})^{2}]$.
\end{lemma}
\begin{proof}
See \cite[Lemma 2.6]{Chen Zhu}.
\end{proof}

\begin{lemma}\label{Lemma 2.12}
Let $r\in R$. If $r+r^{2}\in R^{\mathrm{d}}$ $($resp., $r-r^{2}\in R^{\mathrm{d}}$$)$, then $r\in R^{d}$ and $\mathrm{ind}(r)\leq \mathrm{ind}(r+r^{2})$
$($resp., $\mathrm{ind}(r)\leq \mathrm{ind}(r-r^{2})$$)$.
\end{lemma}
\begin{proof}
See \cite[Lemma 2.7]{Chen Zhu}.
\end{proof}

\begin{theorem}\label{Theorem 2.13}
The following conditions are equivalent for any two projections $p$ and $q$ in a $*$-reducing ring $R$:

\emph{(1)} $pq-qp\in R^{\dag}$;

\emph{(2)} $pq\in R^{\dag}$ and $p-q\in R^{\dag}$.
\end{theorem}

\begin{proof}
(1)$\Rightarrow$(2) Since $b-b^{*} = pq-qp\in R^{\dag}$ and $(b-b^{*})R = (b^{*}-b)R = [(b-b^{*})]^{*}R$,
                    we have $(b-b^{*})^{\dag} = (b-b^{*})^{\sharp}$ by \cite[Proposition 2]{Patricio Puystjens}.
                    It can be verified that $[(b-b^{*})^{\dag}]^{2}$ is the MP inverse of $(b-b^{*})^{2}$.
                    Again, by \cite[Proposition 2]{Patricio Puystjens}, it follows that $[(b-b^{*})^{2}]^{\dag} = [(b-b^{*})^{2}]^{\sharp}$
                    since $[(b-b^{*})^{2}]^*R = (b-b^{*})^{2}R$.
                    In view of Lemma \ref{Lemma 2.11}, we have $bb^{*}\in R^{\mathrm{d}}$ and ind$(bb^{*})$ $\leq$ ind$[(b-b^{*})^{2}] \leq$ 1.

                    Note that $bb^{*} = pq(1-p)qp = pqp-(pqp)^{2}$.
                    By Lemma \ref{Lemma 2.12}, we obtain $pqp\in R^{\mathrm{d}}$ and ind$(pqp) \leq$ ind$(bb^{*})\leq 1$,
                    which means $pqp\in R^{\sharp}$.
                    Hence $(pqp)^{\dag} = (pqp)^{\sharp}$ because $(pqp)^{*}R = pqpR$ (see \cite[Proposition 2]{Patricio Puystjens}).
                    Now, substituting $pq$ for $r$ in Lemma \ref{Lemma 2.1}(2), one can see that $pq\in R^{\dag}$.

                    Thus, $pq-qp\in R^{\dag}$ implies $pq\in R^{\dag}$.
                    Replacing $p$ by $(1-p)$, we get $(1-p)q\in R^{\dag}$ since $(1-p)q-q(1-p) = -(pq-qp)\in R^{\dag}$.
                    In addition, we have $p-q\in R^{\dag}$ by Corollary \ref{corollary 2.8}.

(2)$\Rightarrow$(1) It follows that $(p-q)(p-q)^{\dag} = (p-q)^{\dag}(p-q)$ since $p-q\in R^{\dag}$ and $(p-q)^{*} = p-q$.
                    One can easily check that $[(p-q)^{\dag}]^{2}$Ϊ$(p-q)^{2}$.
                    By Lemma \ref{Lemma 2.1}(1), we have $pqp\in R^{\dag}$ since $pq\in R^{\dag}$.
                    Meanwhile, $[(p-q)^{2}]^{*} = (p-q)^{2}$, $(pqp)^* =pqp$ and $bb^{*} = pqp(p-q)^{2} = (p-q)^{2}pqp$.
                    Combining these facts we can see that $pqp$, $(pqp)^{\dag}$, $(p-q)^{2}$ and $[(p-q)^{2}]^{\dag}$ commute with each other
                    according to \cite[Corollary 12]{Mary}.
                    Now, it is trivial to verify that $(pqp)^{\dag}(p-q)^{\dag}(p-q)^{\dag}$ is the MP inverse of $bb^{*} = pqp(p-q)^{2}$.
                    Moreover, it is clear that $b\in R^{\dag}$ by Lemma \ref{Lemma 2.1}(2).
                    Finally, $pq-qp\in R^{\dag}$ follows by \cite[Theorem 4.1(iv)]{Benitez CvetkovicIlic} (see also \cite[Theorem 13]{Li 2008}).
\end{proof}

\begin{theorem}\label{Theorem 2.14}
The following conditions are equivalent for any two projections $p$ and $q$ in a $*$-reducing ring $R$:

\emph{(1)} $pq+qp\in R^{\dag}$;

\emph{(2)} $p+q\in R^{\dag}$ and $pq\in R^{\dag}$.
\end{theorem}

\begin{proof}
(1)$\Rightarrow$(2) According to \cite[Proposition 2]{Patricio Puystjens},
                    it follows that $(pq+qp)^{\dag} = (pq+qp)^{\sharp}$ since $pq+qp\in R^{\dag}$ and $(pq+qp)^{*}R = (pq+qp)R$.
                    Note that $(p+q)^{2}-(p+q) = (p+q-1)^{2}+(p+q-1) = pq+qp\in R^{\sharp}$.
                    By Lemma \ref{Lemma 2.12}, we have $p+q, p+q-1\in R^{d}$ with
                    ind$(p+q) \leq$ ind$(pq+qp) \leq 1$ and ind$(p+q-1) \leq$ ind$(pq+qp) \leq 1$,
                    i.e., $p+q$, $p+q-1\in R^{\sharp}$.
                    On the other hand, $(p+q)^*R = (p+q)R$ and $(p+q-1)^*R = (p+q-1)R$ imply
                    $(p+q)^{\dag} = (p+q)^{\sharp}$ and $(p+q-1)^{\dag} = (p+q-1)^{\sharp}$ (see \cite[Proposition 2]{Patricio Puystjens}).
                    Finally, by Lemma \ref{lemma 2.10}, we have $pq\in R^{\dag}$ since $(p+q-1)^{\dag}$.

(2)$\Rightarrow$(1) First, $pq\in R^{\dag}$ implies $p+q-1 \in R^{\dag}$ by Lemma \ref{lemma 2.10}.
                    Combining this with the hypothesis $p+q\in R^{\dag}$
                    and the facts $(p+q)^* = p+q$, $(p+q-1)^* = p+q-1$ and $pq+qp = (p+q)(p+q-1) = (p+q-1)(p+q)$,
                    one can see that $p+q$, $p+q-1$, $(p+q-1)^{\dag}$ and $(p+q)^{\dag}$ are commutative with each other
                    by \cite[Corollary 12]{Mary}.
                    Whence it is straightforward to check that $(pq+qp)^{\dag} = (p+q)^{\dag}(p+q-1)^{\dag}$.
\end{proof}
\centerline {\bf ACKNOWLEDGMENTS}
This research is supported by the National Natural Science Foundation of China (11201063),
the Specialized Research Fund for the Doctoral Program of Higher Education (20120092110020),
the Natural Science Foundation of Jiangsu Province(BK2010393) and the Foundation of Graduate Innovation Program of Jiangsu Province(CXZZ12-0082).

\end{document}